\documentclass{article}
\usepackage{amssymb, amsmath, amsthm}
\newtheorem{thm}{Theorem}

\newtheorem{lemma}[thm]{Lemma}

\newtheorem{claim}[thm]{Claim}
\newtheorem{defn}[thm]{Definition}
\newtheorem{cor}[thm]{Corollary}

\DeclareSymbolFont{AMSb}{U}{msb}{m}{n}
\DeclareMathSymbol{\R}{\mathbin}{AMSb}{"52}

\newcommand{\bbr}{\mathbb R}

\newcommand{\cala}{\mathcal A}
\newcommand{\calb}{\mathcal B}

\newcommand{\calr}{\mathcal R}

\newcommand{\cals}{\mathcal S}
\newcommand{\calc}{\mathcal C}
\newcommand{\cali}{\mathcal I}
\newcommand{\calj}{\mathcal J}

\newcommand{\one}{\mathbf{1}}

\newcommand{\shad}{\mathbf{sh}}

\newcommand{\beq}{\begin{equation}}
\newcommand{\eeq}{\end{equation}}
\newcommand{\beqa}{\begin{eqnarray*}}
\newcommand{\eeqa}{\end{eqnarray*}}
\newcommand{\beqan}{\begin{eqnarray}}
\newcommand{\eeqan}{\end{eqnarray}}

\author{Michael Bateman}
\title{ Maximal averages along a planar vector field depending on one variable 
\thanks{This work was supported by NSF grant DMS-0902490.  2010 Mathematics
Subject Classification:  Primary:  42B25, Secondary:   42B20 . }}
\vspace{2ex}
\author{ Michael Bateman
\thanks{Department of Mathematics, UCLA,
Box 951555, Los Angeles, CA 90095-1555 ({\tt bateman@math.ucla.edu.})}  }

\date{}
\begin{document}
\maketitle
\begin{abstract}
We prove (essentially) sharp $L^2$ estimates for a restricted maximal operator associated to a planar vector field that depends only on the horizontal variable.  The proof combines an understanding of such vector fields from earlier work of the author with a result of Nets Katz on directional maximal operators.  
\end{abstract}


\section{Introduction}


We prove an estimate on the $L^2$ norm of a certain maximal operator related to vector fields depending on only one variable.  The author has previously established bounds on the $L^p$ norm of this operator; interpolating these with the $L^2$ bounds in this paper yields (essentially) sharp $L^p$ estimates.  This theorem is loosely related to the problem of bounding Hilbert transforms along a vector field.  Defining the maximal operator requires a bit of notation, which we present below.

\subsection{Averages over rectangles}

We start by defining a maximal operator for any collection of rectangles $\calr $:
\beqa
M _{\calr} f(x) = \sup _{x\in R\in \calr} {1\over
{|R|} } \int _R f.
\eeqa
Let $v\colon \bbr ^2 \rightarrow [0,1]$.  For any rectangle $R$, let $L(R)$ denote
the length of $R$, $w(R)$ the width of $R$, and let $\theta(R)$ be the
interval of width ${w(R) \over {L(R)} }$ centered
at the slope of the long side of $R$.  Now let $V_R = \{p \in R \colon
v(p) \in \theta (R) \}$.  
Next we define the collection of rectangles concerning us.  Fix two parameters 
$0<w\leq 1$ and $0 < \delta \leq 1$, and define 
\beqa
\calr _{\delta} = \{ \text{rectangles $R$ of width  $w $} \colon |V(R)|
\geq \delta |R| \}.
\eeqa
Note that the definition of $\calr _{\delta}$ depends on the vector field $v$; we will suppress this dependence.  
\begin{thm} \label{delta}
Suppose $v\colon \bbr ^2 \rightarrow [0,1]$ depends on one variable, i.e., $v(a,b)=v(a)$.  Then for any
$f\in L^2 (\bbr ^2)$,
\beqa
||M_{ \calr _{\delta} } f || _2 \lesssim   \left( \log {1\over{\delta}}
\right)^{3\over 2} ||f||_2.
\eeqa
\end{thm}
This estimate can be interpolated with the obvious $L^{\infty}$ bound on $M_{\calr_{\delta}}$ to obtain logarithmic bounds when $p\geq 2$.  Additionally, we have the following corollary:
\begin{cor}
Under the same hypotheses as the theorem, when $p\in (1, 2)$ we have for $f\in L^p (\bbr ^2)$,
\beqa
||M_{ \calr _{\delta} } f || _p \lesssim   \left( \log {1\over{\delta}}
\right)^{3 (1- {1\over p})  } {1\over \delta  ^{{2\over p} - 1}  }||f||_p.
\eeqa
\end{cor}
We remark that the theorem here is for rectangles of a fixed width $w$.  It is not clear that the argument here generalizes to the situation of rectangles with arbitrary width.  However, a theorem in that setting may play a role in the study of Hilbert transforms along a one-variable vector field.  Motivation for studying maximal averages comes from differentiation theory; this connection has been known for some time.  More recently, maximal theorems of this flavor (with the density parameter $\delta$) have  been connected to the study of Hilbert transforms along a vector field.  See \cite{LL1}, \cite{LL2} for more on this connection.  

We prove the theorem by combining ideas developed by Nets Katz in the
study of directional maximal operators
(\cite{K1}, \cite{K2}) together with the understanding of one-variable
vector fields obtained by the author in \cite{B1}.  It is likely that the exponent on the logarithm is not sharp, but we do not pursue that idea here.  (For example, the argument to obtain estimate \eqref{offdiagonal} below is rather crude.)  We note however that the operator norm is at least 
$\sqrt{ \log {1\over {\delta } } }$.  This can be seen by considering the slope field $v(x,y)=x$ defined on $[0,1]^2$.  Then for 
$w=\delta$, $\calr _{\delta}$ contains (at least) all rectangles of length $1$ projecting vertically to $[0,1]$ with slope in $[0,1]$.  Now we can construct Kakeya-type sets using rectangles from $\calr _{\delta}$.  Letting $f$ be the characteristic function of such a set shows that 
$||M _{\calr_{\delta } } || _2  \gtrsim \sqrt {   \log  {1\over {\delta } }  } $.  Similarly, the $L^p$ estimates given in the corollary are sharp up to logarithmic factors; this can be seen by again considering $v(x,y)=x$ and letting $f$ be the 
characteristic function of a $\delta \times \delta $ square.  In this setup $M_{\calr_{\delta}} \gtrsim \delta$ on a set of measure approximately one.


\section{Outline of proof} \label{outline}


Recall that all rectangles in question have a fixed width $w$.  
By a standard reduction (see \cite{B1}), we may assume that the slope of each rectangle of length $2^k w$ is in the discrete set
\beqa
S_k = \{ {  {j+ {1\over 2} } \over {2^k} } \colon j \in \{ 0,1,..., 2^k -1\}  \}.
\eeqa

Further, we may assume our ``rectangles" are actually parallelograms projecting to dyadic intervals. 
It will be convenient later to assume all rectangles in $\calr _{\delta}$ live in a bounded region, which we take to be the unit square.  We may do so by (say) approximating with finite subcollections of $\calr _{\delta}$.  
Next we linearize the maximal operator.  That is, for each $x\in \bbr
^2$ we choose a rectangle $R \in   \calr _{ \delta}$ that nearly
achieves the supremum in the definition of the maximal operator.  We
will call this rectangle $\rho(x)$.  It is possible that there is no $R \in \calr _{\delta} $ containing $x$;  let $X$ denote the set of points with this property.  This gives us a map $\rho \colon
\bbr ^2 \setminus X \rightarrow \calr _{\delta}$ and a linear operator defined by 
\beqa
T_{\rho} f(x) = { 1\over {|\rho (x)|} } \int _{\rho (x)} f 
\eeqa
for $x\not\in X$, and $T_{\rho} f(x) =0$ for $x\in X$.  Without loss of generality, we will take this region to be the unit square.
To prove Theorem \ref{delta}, it suffices to prove the same bounds on
the linear operators $T_{\rho}$ independent of the choice function
$\rho$.  To do this, we decompose the operator in a certain way
depending on the vector field and the linearization $\rho$, and apply
the Cotlar-Stein lemma.  From now on, we consider the function $\rho$
to be fixed, and we write $T=T_{\rho}$.  We recall a variant of the Cotlar-Stein
lemma.
\begin{lemma}\label{Cotlar-Stein} [Cotlar-Stein]
Suppose $\{T_j\} _{-\infty} ^{\infty} $ is a sequence of linear
operators acting on a Hilbert space $\mathbf H$, and let $T=
\sum_{j=-\infty} ^{\infty} T_j$.
Assume that $a\colon \mathbb Z \rightarrow \mathbb R$ is such that for all $j, k$,
\begin{equation} \label{pieces}
||T_j T_k ^* ||  \leq a(j-k)
\end{equation}
\beqa
||T_j  ^* T_k || =0.
\eeqa
Then
\beqa
||T|| \leq a(0)^{1\over 2} \left( \sum \sqrt{a(j)} \right) ^{1\over 2}.
\eeqa
\end{lemma}
A straightforward modification of the proof in \cite{D} gives the result claimed here.  The following lemma shows how we will satisfy the hypotheses of the Cotlar-Stein lemma.
\begin{lemma}\label{decomp}
Suppose $v$ is a vector field depending on one variable.  
There exist pairwise disjoint sets $A_1, A_2, A_3, \dots ,$ such that if we define
\beqa
T_j f(x) = \one _{A_j} (x) Tf(x)
\eeqa
for $j=1,2,3,\dots ,$ then for all $j, k$, we have
\begin{equation}\label{diagonal}
||T_j T_k ^*|| \lesssim  \log \left( {1\over {\delta}} \right) ^2
\end{equation}
\begin{equation}\label{offdiagonal}
||T_j T_k ^*|| \lesssim  {1\over {\delta}} 2 ^{-|j-k|} ,
\end{equation}
and 
\beqa
||T^* _j T_k || =0.
\eeqa
\end{lemma}
We will use estimate \eqref{diagonal} when $|j-k| \lesssim \log {1\over
{\delta} }$ and estimate \eqref{offdiagonal} otherwise.
We remark that $||T_j ^* T_k || = 0$ automatically for $j\neq k$, since in this case $ A_j \cap A_k = \emptyset $.  Also, $||T_j ^* T_j || = ||T_j T_j ^* ||$.  With this in mind, Theorem \ref{delta} follows easily from the previous two lemmas by letting 
\beqa
a(n) = C (\log {1\over {\delta}} )^2 \text{   for   } n \leq C \log {1\over {\delta}} ,
\eeqa
\beqa
a( C \log {1\over {\delta}} + n ) = 2^{-n} \text{   for   } n\geq 0,
\eeqa
and
\beqa
a(n) = a(-n)   \text{   for   } n < 0.
\eeqa
Applying the Cotlar-Stein lemma, we see that 
\beqa
||T|| \leq  a(0)^{1\over 2} \left( \sum \sqrt{a(n)} \right) ^{1\over 2}
 \lesssim (\log {1\over {\delta}} )^{3\over 2} .
\eeqa

\subsection{Agenda}
The remainder of the paper is devoted to proving Lemma \ref{decomp}.
In Section \ref{lemmas}, we present the large components of the proof of 
Lemma \ref{decomp} and show how they imply \eqref{diagonal}.  In Section \ref{offdiag}, we show how Lemma \ref{iteration} in Section \ref{lemmas} implies \eqref{offdiagonal}.  In Sections \ref{iterationproof} and \ref{rareproof}, we prove the lemmas from Section \ref{lemmas}.


\subsection{Notation} \label{notation}

If $A$ is a set, we write $\one _A$ to denote the characteristic function of $A$.  We write $C$ to denote universal constants that may vary from one appearance to the next.  We write $x\lesssim y$ to mean $x \leq C y $.  If $\calc$ is a collection of sets, we write 

\beqa
\shad ( \calc ) = \bigcup _{C\in \calc} C.  
\eeqa
$\shad$ stands for ``shadow".
If $A\subseteq \bbr ^2$, we write $\pi _1 (A)$ to denote the projection of $A$ onto the horizontal axis, and 
 $\pi _2 (A)$ to denote the projection of $A$ onto the vertical axis.  

\section{ The main ingredients  } \label{lemmas}
In this section we present the statements of the two most substantial ingredients needed for the proof of 
Lemma \ref{decomp}.  The first concerns a prototype of the operator $M_{\calr _{\delta} }$ defined above, which is closely related to the maximal operator over $\sim {1\over {\delta }}$ arbitrary directions.  A more precise definition is given below.
The second key lemma is the inductive step in a stopping time argument.  It tells us how to define the sets $A_j$ needed for the decomposition of our operator $T$ by identifying the intervals on which rectangles of many different directions might be chosen. 

\subsection{ Statements }
We start with a definition.
\begin{defn} \label{good}
We say a collection of rectangles $\calr$ is \rm good \it if whenever $R_1, R_2 \in \calr $ are such that 
$\pi _1(R_1) = \pi _1(R_2) $, the slope of $R_1$ equals the slope of $R_2$; and if $M_{\calr }$ is weak (1,1).  
\end{defn}
\begin{thm}\label{rare}
Let $v$ be a vector field, and let $N\geq 2$ be an integer.  Suppose $\calr_1 , \calr _2, ..., \calr_N$ are good collections of 
rectangles.  Let $\calr = \cup _{j=1} ^N \calr _j$.  Then
\beqa
||M_{\calr} f ||_{2} \lesssim  \log N ||f|| _2 .
\eeqa
\end{thm}
This theorem is essentially due to Katz, in $\cite{K1}$.  The only difference between this theorem and his is that we allow ourselves to average over $N$ ``good" collections rather than $N$ different directions.  The theorem stated here does not actually follow from Katz's statement, but rather his proof.  We include the proof at the end.  The key point is that if $\calr _j$ is a good collection of rectangles with fixed width, then $M_{\calr _j}$ is weak-type (1,1).  The other part of goodness is more of a convenience.
\begin{lemma}\label{iteration}
Let $I\subseteq [0,1]$ be a dyadic interval.  
Let $E\subseteq I \times [0,1]$ be such that if $x\in E$, then 
$\pi_1 ({\rho (x)}) \subseteq I$.  There exist collections 
$\calr _1, \calr _2, ... \calr _{ {3\over {\delta } } } $, a collection of disjoint intervals $\cali _I$, and sets $E_{good}$ and $E_{bad}$ such that 
\begin{enumerate}
\item $\shad (\cali _I) \subseteq I$,
\item $|\shad (\cali _I)| \leq {1\over 2} |I|$,
\item each $\calr _i$ is a good collection, 
\item  $E_{good} \cap E_{bad} = \emptyset $,
\item $E_{good} \cup E_{bad} = E $,
\item for $x\in E_{good}$, we have $\rho (x) \in \calr _i$ for some $i\in \{1,2,..., {3\over {\delta } } \}$,
\item and for $x\in E_{bad}$, we have $\pi_1 ({\rho (x)})  \subseteq  \shad (\cali _I) $.
\end{enumerate}
\end{lemma}
\subsection{Defining the sets $A_j$ used in the decomposition of $T$}

We now use Lemma \ref{iteration} to construct the sets $A_j$ mentioned in Lemma \ref{decomp}.
The last point in the lemma above guarantees that if $x\in E_{bad}$, then 
$x\in \shad (\cali _I) \times [0,1]$.
 Let 
$\cali_0 = \{ [0,1] \}$ and let $E_0 = [0,1]^2$.  Now suppose we have constructed the collections 
$\cali _0, \cali _1, \dots, \cali _j$ and the sets $E_0,E_1, \dots E_j$.  
For each $I\in \cali _j$, define $E_{j, I} = E_j \cap (I \times [0,1])$.
Apply the previous lemma to the intervals $I\subseteq \cali _j$ with the sets $E_{j,I}$.  Define
\beqa
\cali _{j+1} = \bigcup _{I\in \cali _j} \cali _I,
\eeqa 
and
\beqa
E_{j+1} = \bigcup _{I\in \cali _j} \left( E_{j,I} \right) _{bad},
\eeqa
and
\beqa
A_j = \bigcup _{I\in \cali _j} \left( E_{j,I}\right) _{good}.
\eeqa
These are the sets $A_j$ used in the decomposition of our operator $T$.  By construction, we have for each $I\in \cali _j$ and $k\geq j$, 
\beqa
|I \cap \shad (\cali _k) | \leq 2 ^{-|j-k|} |I|.
\eeqa
This is the key fact needed to prove the estimate (\ref{offdiagonal}) used for $j,k$ far apart.  The proof of 
(\ref{offdiagonal}) occupies Section \ref{offdiag}.


\subsection{Proof of estimate \eqref{diagonal} }


Note that for each fixed $j$, there exist good collections $\calr _1, \calr _2, \dots \calr _{3\over {\delta} }$ such that $\rho (x) \in \cup _{i=1} ^{ {3\over {\delta } } }  \calr _i$ for all $x\in A_j$.  This fact, together with Theorem \ref{rare}, is already enough to establish the estimate \eqref{diagonal}:  Each $T_{j}$ is controlled by a maximal
operator $M_{\calr}$ where 
\beqa
\calr = \bigcup _{i=1} ^{3\over {\delta} } \calr _i
\eeqa
and each $\calr _i$ is a good collection.  Hence we may apply
Lemma \ref{rare} to obtain the estimate
\beqa
||T_j T_k ^*|| _2 \leq ||T_j||_2 ||T_k ||_2 \lesssim \left( \log {1\over {\delta} }\right) ^2.
\eeqa

\section{Proof of the iterative Lemma \ref{iteration}} \label{iterationproof}

We begin by introducing some notation that will help us describe the collection $\cali _I$ in the statement of the lemma.  Recall that all rectangles in question have a fixed width $w$.  
Also recall that the slope of each rectangle is in the discrete set defined at the beginning of Section \ref{outline}
and that our ``rectangles" are actually parallelograms projecting to dyadic intervals.  
For each dyadic interval $J\subseteq I$ and any $s \in S$, we define
\beqa
G_{J,s} = \{ a \in J \colon v(a) 
		\in [s-{w\over {|J|}} , s + {w\over {|J|}} )  \}
\eeqa
\beqa
S(J) = \{ s\in  S \colon |G_{J,s}| \geq \delta |J| \}.
\eeqa
In the rest of this section we will abuse notation 
and write $s$ to denote the dyadic interval centered at $s$.  The convenience of this will be apparent throughout the section.
$S(J)$ is the set of allowable slope for rectangles projecting vertically to $J$.
For $J$ dyadic with $J\subseteq I$, we will define a set of slopes $T(J)$
as follows.  The definition is inductive, starting with the largest
interval and then moving to its subintervals.  First for  $I$,
the largest interval, define
\beqa
T(I) = S(I) . 
\eeqa
Note that $T(I) $ is just the set of allowable slopes for the
interval $I$.
(Recall that the allowable slopes for an interval are those that are
at least $\delta$-popular.)  Now for smaller intervals $J$, we will
define $T(J)$ similarly, except that we will not include slopes that
have been used by an ancestor of $J$.  (By ``ancestor", we mean another
dyadic interval containing $J$.)  More precisely, having defined
$T(K)$ for $K\supsetneqq  J$, define
\beqa
T(J)  =  \{ s \in S(J) \colon s\not\supset s' \text{  for any  } s'
\in T(K), K \supsetneqq J \} .
\eeqa
For $s \in T(J)$, let
\beqa
 \mu _ J^s  = |G_{J,s} |;
\eeqa
otherwise, let $\mu _J ^s = 0$; and let 
\beqa
 \mu _ J = \sum _{s\in T(J)} \mu _J ^s .
\eeqa
It is straightforward to check that 
\begin{equation} \label{carleson}
\sum_{J \subseteq I} \mu _J \leq |I|
\end{equation}
since $G_{J,s} \cap G_{J', s'} = \emptyset$ for $s \in T(J)$ and $s' \in T(J')$ unless 
$J=J' $ and $s=s'$.


We now define the collection $\cali _I$ mentioned in the statement of the lemma.  Let $\cali _I$ be the collection of maximal subintervals $I'$ of $I$ for which 
\beqan \label{calidef}
\sum _{I' \subseteq K \subseteq I }  { {\mu _K } \over { |K| } } \geq 2.
\eeqan
We remark that 
\begin{eqnarray} \label{decay}
| \bigcup \cali _I| & \leq & | \{ a \in I \colon \sum _{ K \subseteq I }  { {\mu _K } \over { |K| } } \one _K (a) \geq 2 \} | \\
& \leq & {1\over 2} |I|
\end{eqnarray}
by Chebyshev's inequality and the Carleson condition \eqref{carleson} .  This proves the second claim of the lemma.  Of course the first claim is true by construction.
Let 
\beqa
\Theta  = \{ (J,s) \colon J \subseteq I \text{  and  } s\in T(J) \},
\eeqa
\beqa
\Theta _{bad} = \{ (J,s)\in \Theta \colon J\subseteq I' \text{  for some  } I' \in \cali _I   \} 
\eeqa
and let
\beqa
\Theta_{good} = \Theta \setminus \Theta _{bad}.
\eeqa
The following partial order on pairs in $\Theta $ will be useful:  we write 
\beqa
(J,s) \leq (J', s') 
\eeqa
whenever either $J=J'$ and the center of $s$ is less than or equal to the center of $s'$, or $J \subsetneqq J'$.  
Note that if 
$J \cap J' \neq \emptyset$, then $(J,s)$ and $(J', s')$ are comparable under the relation $\leq$.  Of course we will write  
$(J,s) < (J', s') $ to mean $(J,s) \leq (J', s') $ but $(J,s) \neq (J', s') $.
 Define the \it children \rm of a pair $(J', s')$ to be all pairs $(J,s) < (J', s')$ that are maximal with respect to this property.  (I.e., there is no pair $(J'', s'')$ such that $(J,s) < (J'', s'') < (J', s')$ .)  Let $C(J,s)$ denote the set of children of $(J,s)$.  Now we sort elements of $\Theta _{good} $ inductively.  Define $\Omega _0$ to be the set of maximal elements of $\Theta _{good} $.  Now having defined $\Omega _0, \Omega _1, \dots , \Omega _{n}$, define
\beqa
\Omega _{n+1} = \Theta _{good} \cap \left( \bigcup _{(J, s) \in \Omega _n }  C(J,s) \right).
\eeqa
Now we let 
\beqa
F_n = \{ x\in E \colon \exists (J,s)\in \Omega_n \text{ with } 
	\pi_1 (\rho (x) ) \subseteq J \text{ and }  slope(\rho (x) ) \supseteq s \},
\eeqa
\beqa
\calr _n = \{ \rho(x)\colon x \in F_n \},
\eeqa
and define
\beqa
E_{good}=  \bigcup _{n=1} ^{\infty } F_n,
\eeqa
\beqa
E_{bad} = E \setminus E_{good}.
\eeqa
This proves claims 4,5, and 6 of the lemma by construction.  
Note that if $slope (\rho (x)) \in S(\pi _1 (\rho (x))$, then there exists $(J,s) \in \Theta$ such that 
$ \pi _1 (\rho (x) \subseteq J$ and $ slope ( \rho (x) ) \supseteq s $.  If $x\in E_{bad}$, then this $(J,s) \not\in \Theta _{good}$.  Hence $(J,s) \in \Theta _{bad}$, so $\pi _1 (x) \in \pi _1 (\rho (x)) \subseteq \shad (\cali _I).$  This proves 7.
To complete the proof of the lemma, it is enough to establish the following two claims:
\begin{claim} \label{alldone}
$\Omega _{3\over {\delta } }$ is empty.  (From this it follows that $F _{3\over {\delta } }$ is empty.)
\end{claim}
\begin{claim} \label{goodn}
For each $n$, $\calr _n$ is a good collection.  
\end{claim} 
Recall that \it good \rm collections are defined in Definition \ref{good}.
\begin{proof} [Proof of Claim \ref{goodn}]
First note that if $slope (\rho (x)) \in S(\pi _1 (\rho (x))$, then there exists $(J,s) \in \Theta$ such that 
$ \pi _1 (\rho (x) \subseteq J$ and $ slope ( \rho (x) ) \supseteq s $.  Now note that if $\calr _0$ is a collection of rectangles such that $slope (R) \supseteq s $ for all $R\in \calr_0$, then $M_{\calr_0}$ is weak (1,1).  This is because all rectangles in $\calr_0$ essentially point in the same direction.  Similarly, if we have a disjoint collection of intervals $\calj$ and 
slopes $\{s_J\}_{J\in \calj}$ such that for each $R\in \calr _0$, we have $J \in \calj$ such that 
$\pi _1 (R) \subseteq J$ and $slope(R) \supseteq s_J$, then $M_{\calr _0}$ is again weak (1,1).  

Hence the claim follows immediately from the following fact:
For every $n=0,1,2,...$,
if $(J_1, s_1), (J_2, s_2) \in \Omega _n$ with 
$(J_1, s_1)\neq  (J_2, s_2) $, then $J_1 \cap J_2 = \emptyset $.

This fact follows from an easy induction argument:  Since $\Omega _0$ contains only maximal elements in $\Theta _{good} $, we cannot have any distinct $(J_1, s_1), (J_2, s_2) \in \Omega _0$ with 
$J_1\subseteq  J_2 $, because in that case it is not possible for both 
$(J_1, s_1)$ and $ (J_2, s_2) $ to be maximal.  Now suppose the claim is true for distinct pairs in
$ \Omega _{n-1}$, and suppose $(J_1, s_1), (J_2, s_2) \in \Omega _n$.  By definition of $\Omega _n$, there exist $(J' _1 , s' _1), (J' _2, s' _2) \in \Omega _{n-1}$ such that $(J_i, s_i ) \leq ( J' _i , s' _i ) $ for $i=1,2$.  This implies, in particular, that $J_i \subseteq J' _i$ for $i=1,2$.  By our induction hypothesis, we know that either 
$J' _1 $ does not intersect $J' _2 $, or $(J' _1 , s' _1) = (J' _2, s' _2) $.  In the first case, it is obvious that $J_1 \cap J_2 = \emptyset $.  In the second case, we argue as we did in the $n=0$ case:  if, say,  $J_1 \subseteq J_2$, then it is not possible for both of $(J_1, s_1)$ and $ (J_2, s_2) $ to be maximal children of $(J' _1 , s' _1) $.  This proves the claim.
\end{proof}

\begin{proof} [Proof of Claim \ref{alldone}]
We begin by defining, for any dyadic $K\subseteq I$, 
\beqa
\Theta _K = \{ (J,s) \colon K\subseteq J\subseteq I \text{ and } s\in T(J) \}.
\eeqa
Note that if $s\in T(J)$, then ${ {\mu ^s _J } \over {|J|} } \geq  \delta $, so 
\beqa
 \# \left( \Theta _K \right)  &=& \sum _{K\subseteq J\subseteq I } \# \left( T(J) \right)  \\
 &=&   \sum _{K\subseteq J\subseteq I } \sum _{s \in T(J) } 1 \\
& \leq & {1\over {\delta } }  \sum _{K\subseteq J\subseteq I } \sum _{s \in T(J) } { {\mu ^s _J } \over {|J|} } \\
& = &  {1\over {\delta } }  \sum _{K\subseteq J\subseteq I }    { {\mu  _J } \over {|J|} } .
\eeqa
If the claim were false, then there would be a sequence
\beqa
(J_1, s_1) > (J_2, s_2) > \dots > (J_{3\over {\delta } } , s_{3\over {\delta } } ) 
\eeqa
with $(J_i, s_i) \in \Theta _{good}$ for $i=1,2, \dots , {3\over {\delta } }$.  But this implies 
\beqa
{3\over {\delta } } \leq \# \left (\Theta _{J_{3\over {\delta } } } \right) 
	\leq {1\over {\delta } } \sum _{ J_{3\over {\delta } } \subseteq K \subseteq  I } { {\mu _K } \over { |K| } },
\eeqa
which is impossible since $J_{3\over {\delta } } \not\subseteq I'$ for any $I' \in \cali _I$.  See the definition of 
$\cali _I$ in \eqref{calidef}.
This proves the claim.
\end{proof}



\section{Proof of estimate  \eqref{offdiagonal} } \label{offdiag}


In this section, we establish the estimate \eqref{offdiagonal}.  Recall that we use this estimate when $|j-k|$ is rather large.  To prove it, we take advantage of the rapid decay of $|J \cap \shad (\cali _k ) |$ whenever $J\in \cali _j$ and $k$ is much larger than $j$.  Because of this decay, we have that rectangles chosen by points in $A_k$ will only be able to intersect rectangles $R$ chosen by points in $A_j$ on very small subsets of $R$.  Essentially all of the analysis of this section takes place on a fixed interval $J\in \cali _j$.  We formalize these ideas below.  In this section, we use notation from Sections \ref{lemmas} and \ref{iterationproof}.  The reader may wish to ignore the dependence on $J$ in some of the notation below and imagine that $\cali _j$ consists of a single interval.

Note that $T_j T_k^* = (T_k T_j ^* ) ^*$, so it is enough to control
$|| T_k T_j ^* ||$ in the case $j\leq k$.  So fix $j$ and $k$ with $j \leq
k$.  Recall that
\beqa
T_j f(x) = \one _{A_j} (x) {1\over {\rho (x) } } \int _{\rho (x)} f,
\eeqa
where $\rho $ is a fixed linearizing function.  Fix $J \in \cali _j$.  By Lemma \ref{iteration} and the definition of $A_j$ following the lemma, we know there are collections 
$\calr ^j _1, \calr ^j _2, \dots \calr ^j _{3\over {\delta } }$ such that if $x\in A_j$, then $\rho (x) \in \cup _n \calr ^j _n$.  Further, $\pi (\rho (x)) \subseteq \shad(\cali _j) $ for all $x\in A_j$.  Let
\beqa
A_{j,J} = \{ x\in A_j \cap ( J \times [0,1] )  \}
\eeqa
and for each $n=1,2,\dots , {3\over {\delta } }$, let
\beqa
A_{j,J,n} =\{ x\in A_{j,J} \colon \rho (x) \in \calr ^j _n \}.
\eeqa
With this notation, if $x\in A_{j,J}$ we define
\beqa
T_{j, J, n} f(x) = \one _{A_{j,J, n}} (x)  T_j f(x)
\eeqa
and
\beqa
T_{j, J} f(x) = \sum_{n=1} ^{3\over {\delta } } T_{j, J, n} f(x).
\eeqa
Note that
\beqa
T_j = \sum _{J\in \cali _j} T_{j, J},
\eeqa
and that $T_{j,J} ^*f=T^* _{j,J} (f \one _{J\times [0,1]} )$ is supported on $J\times[0,1]$.
To prove the estimate \eqref{offdiagonal}, it is enough to prove
\begin{equation} \label{piece}
||T_k T_{j, J, n}^* || \lesssim  2^{-|j-k|}
\end{equation}
for every $J\in \cali _j$ and every $ n \in \{1,2,\dots, {3\over {\delta } } \}$, because then
\beqa
|| T_k T_j ^* f||_2 
	& \leq & \sum _{J\in \cali _j} \sum_{n=1} ^{3\over {\delta } } 
		||T_k T_{j, J, n} ^* (f \one_{J\times{[0,1]} } )||_2 \\
& \leq &  2^{-|j-k|} \sum _{J\in \cali _j} 
		\sum_{n=1} ^{3\over {\delta } }  ||(f \one_{J\times{[0,1]} } )|| _2 \\
& \lesssim &  2^{-|j-k|} {1\over {\delta } } ||f||_2.
\eeqa
To prove \eqref{piece}, and hence \eqref{offdiagonal}, it is enough to prove the following two
claims:
\begin{claim} \label{firstprop}
For each $n=1,2,\dots {3\over {\delta } }$, and each $x\in A_k$, 
\beqa
T_k T_{j, J, n}^* f(x) \leq M_2 T_{j, J, n}^* f(x),
\eeqa
where $M_2$ is the standard Hardy-Littlewood maximal operator along vertical line segments.  
\end{claim}

\begin{claim} \label{secondprop}
For $\lambda > 0$,
\beqa
|\{ x\in A_k \colon M_2 T_{j, J, n}^* f(x) > \lambda \}
       \lesssim 2^{-|j-k|} |\{x\in \bbr ^2 \colon M_2 T_{j, J, n}^* f(x) > \lambda \}|.
\eeqa
\end{claim}
With these two claims, we see that
\beqa
||T_k T_{j, J, n} ^* f ||^2 _2 & \lesssim &
\int _0 ^{\infty} \lambda 2^{-|j-k|} |\{x\in \bbr ^2 \colon M_2 T_{j,
J, n}^* f(x) > \lambda \}| d\lambda \\
&=& 2^{-|j-k|} || M_2 T_{j, J, n}^* f || \lesssim 2^{-|j-k|} ||f||_2 ^2,
\eeqa
since $M_2$ and $T_{j,J,n}$ are bounded on $L^2$ with uniform constants, which proves estimate \eqref{piece}, and hence the estimate \eqref{offdiagonal}.  ($T_{j,J,n}$ is bounded because each $\calr _n ^j $ is a good collection.)
We turn to the proofs of these two claims.  
\begin{proof} [Proof of Claim \ref{firstprop} ]
Fix any $K \in \cali _k$ such that $K\subseteq J$.  (If $x \notin J \times [0,1]$, then 
$T_k T_{j, J, s}^* f(x)=0.$)  There exists $J' \in \cali _{j+1}$ with 
$K \subseteq  J' \subseteq J$.  Note that all $R \in \calr ^j _n$ with 
$\pi_1(R) \supseteq J'$, have the same slope.  For suppose such $R_1$, $R_2$ have different slopes.  Then by the fact mentioned at the beginning of the proof of Claim \ref{goodn}, and by the definition of the sets $F_n$ given in the last section, we know 
$\pi_ 1(R_1 ) \cap \pi_1 (R_2) = \emptyset $.  But this contradicts the claim that 
$\pi_1(R_1)$ and $\pi _1(R_2)$ both contain $J'$.

Hence all $R \in \calr ^j _n$ with 
$\pi_1(R) \supseteq J'$, have the same slope; let's call it $\theta$.  This implies that 
$T_{j, J, n}^* f(x)$ is constant along line segments contained in $J'$ with slope 
$\theta$.  Let $Y_{\theta}$ be any line orthogonal to a line segment with slope 
$\theta$.   Because $T_{j, J, n}^* f(x)$ is constant along line segments contained in 
$J' \times [0,1] $ with slope $\theta$, we know that if $L$ is a line segment contained in $J' \times [0,1] $, then 
\beqa
{1\over {|L|}} \int _L T_{j, J, n}^* f(x) 
\eeqa
depends only on the projection of $L$ onto the axis $Y_{\theta}$, and in particular, it does not depend on the slope of $L$.  (Of course the integral here is with respect to one-dimensional Lebesgue measure.)  Hence
\beqa
{1\over {|L|}} \int _L T_{j, J, n}^* f(x)  \leq M_2 T_{j, J, n}^* f(x).
\eeqa
Since $T_k$ is essentially an average over line segments, this finishes the proof of the claim.
\end{proof}
\begin{proof} [Proof of Claim \ref{secondprop} ]
The set $A_k \cap (J \times [0,1] ) $ is supported on the set
\beqa
\bigcup _{J' \in \cali_{j+1} } (J' \times [0,1] ) .
\eeqa
Hence it suffices to prove that for any $J' \in \cali_{j+1}$, we have
\beqa
|\{ x\in A_k\cap (J' \times [0,1] ) & \colon & M_2 T_{j, J, n}^* f(x) > \lambda \}| \\
& \lesssim & \\
 2^{-|j-k|} |\{x\in (J' \times [0,1] ) & \colon & M_2 T_{j, J, n}^* f(x) > \lambda \}| .
\eeqa
So we fix attention on a particular $J' \in \cali_{j+1}$.  By the argument in the proof of the previous claim, we know that all $R \in \calr ^j _n$ with 
$\pi_1(R) \supseteq J'$, have the same slope $\theta$, which implies that 
$T_{j, J, n}^* f(x)$ is constant along line segments contained in $J'$ with slope 
$\theta$.  This further implies that $M_2 T_{j, J, n}^* f(x)$ is constant along segments of length $|J'|$ with slope $\theta$.  But since 
\beqa
|J' \cap \shad (\cali _k ) | \lesssim 2^{-|j-k|} |J'|,
\eeqa
we have proved the claim.
\end{proof}


\section{Proof of Lemma \ref{rare} } \label{rareproof}


In this section, we prove Theorem \ref{rare}.  The argument given here is due to Katz \cite{K1}.  Recall that we assume $\calr_1 , \calr _2, ..., \calr_N$ are good collections of
rectangles and $\calr = \cup _{j=1} ^N \calr _j$.  To prove the theorem, we prove the weak-type estimate
\beqa
|\{ M_{\calr } f > \lambda \} | \lesssim \log N  { { ||f||^2 _2 } \over {\lambda ^2 } }.
\eeqa
To prove the weak-type bound above, we linearize the maximal operator as above, and prove restricted strong-type bounds for the linearization.  That is, we prove
\beqa
|| T^* \one _E || ^2 _2 \lesssim \log N |E|
\eeqa
for any set $E$, where again we write $T$ to denote a particular linearization of $M_{\calr}$.  As before, we will let $\rho \colon [0,1]^2 \rightarrow \calr $ denote the linearization.  Of course the estimates are independent of the particular linearization.    

To upgrade this weak-type estimate to the desired strong-type estimate, one only needs to apply standard interpolation theorems.  (Since $T$ is trivially bounded on $L^{\infty}$, we may interpolate to obtain strong-type estimates for $p>2$.  Then interpolate with the trivial weak (1,1) estimate of 
$ \sim N$ to obtain the claimed strong bounds on $L^2$.)

For the rest of the section, we focus on proving this retricted strong-type estimate for $T^*$.  It is convenient to assume that $\pi_1 (R) $ is a dyadic interval; we do so.
For a set $F$ and any interval $I$, let 
\beqa
F_I = \{ x\in F \colon \pi_1 (\rho (x) ) \subseteq I \}.
\eeqa
For any rectangle $R$ and any set $F$, define 
\beqa
B ^F _R = {1\over {|R|}} \int _R T^*( \one _{F_{ \pi _1 (R)} } ).
\eeqa
and
\beqa
\nu ^F  _R = | \{ x\in F \colon \rho (x) = R \} |.
\eeqa
(Recall that $\pi _1 (R)$ is the projection of $R$ onto the horizontal axis.)
The quantity $B_R$ is called the \it badness \rm of the rectangle $R$.
Before we proceed any further, we present one computation that is crucial for understanding this section.
\begin{claim}
For a set $F$, 
\beqa
T^* (\one _F ) (x) =  \sum _{R\in \calr}   { {  \nu^F _R \one _R (x) } \over {|R| } }.
\eeqa
\end{claim}
This is a weighted count of the rectangles in $\calr$ that contain $x$.  
\begin{proof}
Recall that 
\beqa
T f (x) = \int { { \one _{\rho (x)} (y) } \over {|\rho (x) | } } f (y) dy.
\eeqa
This means that 
\beqa
T^* (\one _F ) (x) &=& \int { { \one _{\rho (y)} (x) } \over {|\rho (y) | } } \one _F (y) dy \\
   & = & \sum _{R\in \calr} \int _{\{y \colon \rho (y) = R\}}  { { \one _{\rho (y)} (x) } \over {|\rho (y) | } } \one _F (y) dy  \\
   & = &  \sum _{R\in \calr}   { { \one _R (x) } \over {|R| } } \int _{\{y \colon \rho (y) = R\}}    \one _F (y) dy  \\
   & = &  \sum _{R\in \calr}   { {  \nu^F _R \one _R (x) } \over {|R| } }.
\eeqa
\end{proof}
An immediate corollary of this is the estimate
\beqan \label{newcarleson}
\int T^* (\one _F )  = \sum _{R\in \calr} \nu^F _R \leq |F|.
\eeqan
Because of this computation, we see that the badness $B_R$ is a weighted count of the rectangles $R'$ that intersect $R$ and that are essentially shorter than $R$.  The weighting depends on the measure $\nu ^F _{R'}$, the length of $R'$, and the angle between $R$ and $R'$.  (If $R$ and $R'$ intersect with smaller angle, then $\one _{R'} (x)$ will be supported on a larger portion of $R$.)

\begin{claim} \label{reformulate}
\beqa
\int \left( T^* (\one _E ) \right) ^2 \lesssim \sum _{R\in \calr} \nu^E _R B_R.
\eeqa
\end{claim}
Because of this claim and the definition of $B_R$, we see that $||T^* (\one _E) ||^2 _2 $ is essentially a count of quantities like 
$|R_1 \cap R_2|$, which is to be expected in an $L^2$ estimate of this operator.
This claim follows from a straightforward computation which we carry out shortly.  We will combine it with the following lemma to prove the theorem.

\begin{lemma} \label{decaylemma}
Let $\calr _0$ be a good collection of rectangles.  Let 
\beqa
\cals _k = \{ R \in \calr _0 \colon B_R \in [k-1, k) \}.
\eeqa
Then 
\beqa
|\bigcup _{R \in \cals _{k}} R | \lesssim 2^{-ck} |E|.
\eeqa
\end{lemma}
We can already use these facts to prove the theorem of this section.  
Let 
\beqa
\cals  _{n,k} = \{ R \in \calr _n \colon B_R \in [k-1, k) \}.
\eeqa 
By the claim, 
\beqa
\int \left( T^* (\one _E ) \right) ^2 &\lesssim & \sum _{R\in \calr} \nu ^E _R B_R \\
 &=& \sum _{n=1} ^{N} \sum _{R\in \calr _n} \nu^E _R B_R \\
&= & \sum _{k=1} ^{\infty} \sum _{n=1} ^{N } 
				\sum _{R\in \cals _{n,k} } \nu^E _R B_R \\
&\lesssim& \sum _{k=1} ^{\infty}  k \sum _{n=1} ^{N } 
				\sum _{R\in \cals _{n,k} } \nu^E _R \\
&=& (\star)
\eeqa
Note that $\sum _{R\in \cals _{n,k} } \nu^E _R  \leq |\bigcup \cals _{n,k} |$ and that 
$\sum _{R\in\calr} \nu _R ^E \leq |E|$.  This first estimate is useful when $k$ is large, and the second when $k$ is small.
So by the lemma,  
\beqa
(\star) &\lesssim & \sum _{k=1} ^{\sim \log N}  \log N \sum _{n=1} ^{N } 
		\sum _{R\in \cals _{n,k} } \nu^E _R  
    + \sum _{k \sim \log N} ^{\infty} k \sum _{n=1} ^{N } 
		\sum _{R\in \cals _{n,k} } \nu^E _R  \\
&\lesssim & \log N |E| + \sum _{k \sim \log N} ^{\infty} k N  2^{-ck} |E| \\
&\lesssim & \log N |E|.
\eeqa 
This proves Theorem \ref{rare} modulo Claim \ref{reformulate} and Lemma \ref{decaylemma}.   First we prove Claim \ref{reformulate}.
\begin{proof}[Proof of Claim \ref{reformulate}]
The proof is a straightforward computation: 
\beqa
\int \left( T^* (\one _E ) \right) ^2 
 &=& \int \left( \sum _{R\in \calr}   { {  \nu^F _R \one _R (x) } \over {|R| } } \right) ^2 \\
&=& \int \sum _{R\in \calr} \sum _{Q \in \calr} 
     { {  \nu^F _R \one _R (x) } \over {|R| } } { {  \nu^F _Q \one _Q (x) } \over {|Q| } } \\
&\lesssim & \sum _{R\in \calr}  \nu^F _R { 1 \over {|R| } } 
	\int _R \sum_{Q\in \calr \colon \pi_1 (Q)\subseteq \pi_1 (R) } 
		{ {  \nu^F _Q \one _Q (x) } \over {|Q| } },
\eeqa
where we have used symmetry to restrict the sum in the final integral to rectangles $Q$ that are essentially shorter than $R$.  To finish the proof, we need only note that 
\beqa
\sum_{Q\in \calr \colon \pi_1 (Q)\subseteq \pi_1 (R) } 
		{ {  \nu^F _Q \one _Q (x) } \over {|Q| } } = 
	T^* (\one _{F_{\pi_1 (R) } } ).
\eeqa
\end{proof}
To prove Lemma \ref{decaylemma}, we iterate the following key lemma.  
\begin{lemma}\label{key}
Let $E$ be a set.  Let $\calr _0$ be a good collection of rectangles.  There exists a set $E' $ such that $|E'| \leq {1\over 2} |E|$ and such that for all 
$R \in \calr $, either
\beqa
B_R ^E \leq C
\eeqa
or
\beqa
R \subseteq E' \text{  and  } B_R ^E \leq C + B_R ^{E'} .
\eeqa
Here $C$ is a universal constant.  
\end{lemma}
\subsection{Proof that Lemma \ref{key} implies Lemma \ref{decaylemma}  }
Given Lemma \ref{key}, we proceed as follow.  Define $E_0= E$.  Apply the lemma to find $E_1$ with $|E_1| \leq {1\over 2} |E_0|$ such that for every $R\in\calr _0$, either 
\beqa
B_R ^E \leq C,
\eeqa 
or 
$R\subseteq E'$ and 
\beqa
B_R ^{E_0} \leq B_R ^{E_1} + C.
\eeqa
Repeat to find $E_2, E_3, E_4, \dots$, with $|E_{j+1}| \leq {1\over 2} |E_{j}|$ such that for every $R\in\calr _0$, either 
\beqa
B_R ^{E_j} \leq C
\eeqa
 or $R \subseteq E_{j+1}$ and
\beqa
B_R ^{E_j} \leq B_R ^{E_{j+1} } +C .
\eeqa 
Now suppose $R$ is such that 
\beqa
B_R ^{E_0} \geq Ck
\eeqa
 for some integer $k\geq 1$.  Then we know $R\subseteq E_1$ and 
\beqa
B_R ^{E_0} \leq C + B_R ^{E_1}.
\eeqa
This implies
\beqa
B_R ^{E_1} \geq C (k-1).
\eeqa
Similarly, this implies $R\subseteq E_2$, $B_R ^{E_0} \leq 2C + B_R ^{E_2}$, and 
\beqa
B_R ^{E_2} \geq C (k-2).
\eeqa
Iterating, we see that $R \subseteq E_{k-1}$.  This implies 
\beqa
| \bigcup _{ R \in \cals _{Ck}} R| \lesssim 2^{-k} |E|,
\eeqa
which proves Lemma \ref{decaylemma}.

\subsection{Proof of Lemma \ref{key} }

Fix any intervals $I, K$.  Without loss of generality, we assume that the rectangles in $\calr _0$ that project vertically to $I$ have slope zero.  This is a notational convenience only.  We need some notation to define the set $E'$ from the statement of the lemma.  In the following definitions, dependence on the set $E$ is suppressed.  Define
\beqa
\cala _{I, K} ^{\text{in}} = \{ R \colon \pi_1(R) \subseteq I \text{ and }  \pi_2 (R) \subseteq 3K \},
\eeqa
\beqa
\cala _{I, K} ^{\text{out}} = \{ R \colon \pi_1(R) \subseteq I \text{ and }\pi_2 (R) \nsubseteq 3K \},
\eeqa
\beqa
E_{I,K} ^{\text{in}} = \{ x\in E \colon \rho (x) \in \cala _{I, K} ^{\text{in}} \}, 
\eeqa
\beqa
E_{I,K} ^{\text{out}} = \{ x\in E  \colon \rho (x) \in \cala _{I, K} ^{\text{out}} \} ,
\eeqa
\beqa
B_{I,K} ^{\text{in}} = {1 \over {|I||K| } } \int _{I\times K} T^* (\one _{E_{I,K} ^{\text{in}}  } ),
\eeqa
\beqa
B_{I,K} ^{\text{out}} = {1 \over {|I||K| } } \int _{I\times K} T^* (\one _{E_{I,K} ^{\text{out}}  } ).
\eeqa

Note that for any axis parallel rectangle $R$ with $\pi_1 (R) = I$, we have
\beqa
B_R ^E = {1\over {|R|}} \int _R T^* (\one _{E_{I,K} ^{\text{in}} } ) + {1\over {|R|}} \int _R T^* (\one _{E_{I,K} ^{\text{out}} } )
\eeqa
for any interval $K$.


Let $\calb _I$ be the collection of intervals $K$ such that 
\beqa
B_{I,K} ^{\text{out}} \geq \lambda _0, 
\eeqa
but such that 
\beqa
B_{I,3K} ^{\text{out}} < \lambda _0
\eeqa
 where 
$\lambda _0 \geq 1$ is a universal constant to be specified later.  
We now define the set $E'$ from the statement of Lemma \ref{key}: let
\beqa
E'=\bigcup _{ I } \bigcup _{K \in \calb _I } ( I\times 3 K ).
\eeqa
We also define the the auxiliary set 
\beqa
F = \{x \colon M_{\calr_0} T^* \one _E \geq {{\lambda _0}\over 2} \}.
\eeqa
We will show that $|E'|\leq C|F|$ and then that $|F|\leq {{|E|} \over {2C}}$.  To prove the second estimate we need only the weak (1,1) estimate for $M_{\calr _0} $.  To prove the first estimate we need the following claim.
\begin{claim} \label{universal}
If $K \in \calb _I$, and $\pi_1 (R) = I$ and $\pi _2 (R) \subseteq K$, then 
\beqa
{1\over {|R|}} \int _R T^* (\one _{E_{I,K} ^{\text{out}} }) \leq 20 \lambda _0.
\eeqa 
\end{claim}
\begin{proof}
We will show that for any $a\in K$, 
\beqa
{1\over {|I|}} \int _{I\times a} T^* (\one _{E_{I,K} ^{out} } ) \leq 20 \lambda_0,
\eeqa
which implies the claim.  (Of course the integration immediately above is with respect to $1$-D Lebesgue measure.)  
Note that if $R\in \cala_{I,K}^{out}$ and $R$ intersects $I\times a$, then 
$R$ intersects $I\times a'$ for $a'$ in a set of measure ${1\over 3} |3K|$.  This implies that if 
$R\in \cala_{I,K}^{out}$, then 
\beqa
{1\over {|I||3K|}} \int_{I\times 3K} \one _R 
	\geq {1\over {10}} {1\over {|I|}} \int_{I\times a} \one _R 
\eeqa
for every $a\in K$.  Also note that 
\beqa
\int T^*(\one _{E_{I,K} ^{out}  \setminus  E_{I,3K} ^{out} }) 
\leq \int T^* (\one _{I\times 3K} )
\leq |I||3K|
\eeqa 
by \eqref{newcarleson}, because $T^*$ is positive and 
because $E_{I,K} ^{out}  \setminus  E_{I,3K} ^{out} \subseteq I\times 3K$.
Combining this with the fact that $B_{I, 3K} ^{out} < \lambda $, we know 
\beqa
{1\over {|I|}} \int _{I\times a} T^* (\one _{E_{I,K} ^{out} } ) 
  & \leq & 10  {1\over {|I||3K| } } \int _{I \times 3K} T^* (\one _{E_{I,K} ^{out} } ) \\
& = &   10  {1\over {|I||3K| } } 
	\int _{I \times 3K} T^*(\one _{E_{I,K} ^{out}  \setminus  E_{I,3K} ^{out} }) \\
&+&	10 {1\over {|I||3K| } } \int _{I \times 3K} T^* (\one _{E_{I,3K} ^{out} }) \\
& \leq & 10 + 10{1\over {|I||3K| } } \int _{I \times 3K} T^* (\one _{E_{I,3K} ^{out} } )\\
& \leq & 20 \lambda _0.
\eeqa
\end{proof}
Consider an interval $I$ with $K\in\calb _I$.
By the proof of the previous claim, we know that there exists $A \subseteq K$, such that 
$|A| \geq {1\over {20}} |K|$, and such that 
\beqa
{1\over {| I | } } \int _{I \times a} T^* ( \one _{E_{I,K} ^{out} } ) 
	\geq {{\lambda _0 } \over 2}
\eeqa
for all $a\in A$, where here the integral is taken with respect to one-dimensional measure on  $I \times a$.  This is because the proof gives an upper bound on such averages; this, together with the lower bound on $B_{I,K} ^{\text{out}}$ yields the claimed lower bound for many $a \in A$.
Hence if $a\in A$, then for all $x \in I \times a$, we have
\beqa
M_{\calr _0}  T^* (\one _E ) (x) \geq {{\lambda _0} \over 2}.
\eeqa
This implies that if $x\in (I\times 3K)$, then 
\beqa
M_2 \one _F (x) \geq {1\over {40}}.
\eeqa
Hence
\beqa
|E'| &=&
\left| \bigcup _{I } \bigcup _{K \in \calb _I } ( I\times 3 K ) \right| \\
	& \leq & |\{ x\colon M_2 \one _F (x) \geq {1\over {40} } \} | \\
	& \leq & C |F| \\
	& = & C | \{x \colon M_{\calr _0} T^* \one _E \geq {{\lambda _0} \over 2} \} | \\ 
	& \leq & {C \over {\lambda _0 } } || ( T^* \one _E ) || _1 \leq {1\over 2} |E|,
\eeqa
provided $\lambda _0$ is large enough.  Here we used the fact that $\calr _0$ is a good collection of rectangles.  
%

We just proved the claim about the size of $E'$.
If $B_R \geq  \lambda _0$, then there exists $K$ such that $K \in \calb _{\pi _1 (R)}$.  Further, by Claim \ref{universal} and the positivity of $T^*$, 
\beqa
B_R &=& {1\over {|R|}} \int _R T^* \one _{E_{I,K} ^{\text{in}} } 
		+ {1\over {|R|}} \int _R T^* \one _{E_{I,K} ^{\text{out}} } \\
	& \leq & B_R ^{E'}
		+ 20 \lambda _0 , \\
\eeqa
which proves the other claim of Lemma \ref{key}.

%

\end{document}